\documentclass[12pt, reqno]{amsart}
\usepackage[top=3cm, bottom=3.5cm, left=3cm, right=3cm]{geometry}
\usepackage{hyperref,amsmath,amssymb,amsthm,minitoc}
\usepackage[capitalize,nameinlink,noabbrev]{cleveref}
\usepackage{mathtools}
\usepackage[utf8]{inputenc}
\usepackage{hyperref}
\usepackage{mathrsfs}
\usepackage{lipsum}
\newcommand{\ddb}{\partial\bar\partial}
\newcommand{\tr}{\operatorname{tr}}

\newcommand{\Ric}{\operatorname{Ric}}
\newcommand{\minn}{_\text{min}}
\newcommand{\maxx}{_\text{max}}

\usepackage{geometry}
\usepackage{booktabs}
\usepackage{array}
\usepackage{paralist}
\usepackage{verbatim}
\usepackage{subfig}
\usepackage{fancyhdr}
\author{Xi Sisi Shen*}
\thanks{*Supported in part by NSF grant DMS-2601275.}
\address{Department of Mathematics\\
  CUNY City Tech\\
  New York, NY 11201, USA}
\email[X. S. Shen]{xi.shen15@citytech.cuny.edu}
\author{Kevin Smith}
\address{Department of Mathematics\\
  Imperial College London\\
  180 Queen's Gate, London SW7 2AZ, UK}
\email[K. Smith]{kevin.smith1@imperial.ac.uk}

\newtheorem{proposition}{Proposition}
\newtheorem{theorem}{Theorem}

\newtheorem{lemma}{Lemma}
\newtheorem{remark}{Remark}

\numberwithin{equation}{section}

\title{Coupled continuity equations for constant scalar curvature K\"ahler metrics}

\begin{document}

\baselineskip=15pt

\bibliographystyle{amsplain}

\begin{abstract}
    Inspired by a parabolic system of Li-Yuan-Zhang and the continuity equation of La Nave-Tian, we study a system of elliptic equations for a K\"ahler metric $\omega$ and a closed $(1, 1)$-form $\alpha$. Assuming a uniform estimate for $\omega$, we prove higher order estimates and smooth convergence to a cscK metric coupled to a harmonic $(1, 1)$-form. A simplification of the system is used to recover existence results for K\"ahler-Einstein metrics when $c_1(X) < 0$. On Riemann surfaces with genus at least $2$, we show smooth convergence to the unique K\"ahler-Einstein metric from a large class of initial data.
\end{abstract}
\maketitle

\section{Introduction}\label{section:introduction}

\noindent Let $X$ be a compact K\"ahler manifold of complex dimension $n$. The search for canonical metrics within a given K\"ahler class was initiated by Calabi in his seminal paper \cite{calabi1}. He introduced the functional
\begin{equation*}
    \operatorname{Cal}(\omega) = \int_X R^2 \, \omega^n,
\end{equation*}
where $R$ is the scalar curvature of $\omega$, whose critical points are called extremal metrics and whose minimizers are constant scalar curvature K\"ahler (cscK) metrics. Earlier, the special case of K\"ahler-Einstein metrics was famously solved by Aubin \cite{aubin78} when $c_1(X) < 0$ and Yau \cite{yau78} when $c_1(X) < 0$ or $c_1(X) = 0$. The existence problem for K\"ahler-Einstein metrics reduces to a second order fully nonlinear elliptic PDE called the complex Monge-Amp\`ere equation. In contrast, the general cscK problem is equivalent to a fourth order PDE, and thus the situation is much more challenging. Recently, Chen-Cheng \cite{chen2021constant, chen2021constant2, chen2018constant} made a large breakthrough for the cscK problem where they recast the fourth order equation as a system of second order equations
\begin{align*}
    F &= \log\frac{\omega_\varphi^n}{\omega^n}\\
    \Delta_\varphi F &= -\underline{R}+\tr_\varphi\Ric(\omega).
\end{align*}
The Calabi flow was introduced by Calabi for K\"ahler metrics in \cite{calabi1,calabi2} as a parabolic approach to studying the existence of cscK metrics. The flow is given by
\begin{equation*}
    \partial_t \omega(t) = i \partial \bar \partial R(\omega(t))
\end{equation*}
and has cscK metrics as its stationary points. The Calabi flow on Riemann surfaces was studied by Crusciel in \cite {chrusciel} where long-time existence and convergence to a cscK metric were shown. Chen provided another proof on Riemann surfaces using a concentration/compactness argument \cite{chen01}. Chen-He \cite{chen2008calabi} showed that this flow exists as long as the Ricci curvature remains bounded. Sz\'ekelyhidi \cite{szekelyhidi12} show that the flow converges to a constant scalar curvature K\"ahler metric if the curvature tensor remains uniformly bounded and the Mabuchi energy is proper. In \cite{chen-sun10}, Chen-Sun show stability of the flow, that is, if the initial metric is very close to the cscK metric, then the flow exists and converges uniformly to the cscK metric. More results on the Calabi flow can be found in \cite{bv20, bdl, chang, ChenHe2012, fine10, he15, huangarxiv, huang15, hz12, lwz, streets14, streets16, struwe02, tw07}. 

In \cite{li2020new}, Li-Yuan-Zhang  introduced the parabolic system
\begin{align}
    \partial_t \omega & = -\Ric(\omega) + \lambda \omega + \alpha\label{parabolicsystem1}\\
    \partial_t \alpha & = \Delta \alpha.\label{parabolicsystem2}
\end{align}
It is a modified K\"ahler-Ricci flow for a K\"ahler metric $\omega$ coupled with the heat flow for a closed $(1, 1)$-form $\alpha$, and its stationary points are pairs consisting of a cscK metric with a harmonic $(1, 1)$-forms. Thus, like the Chen-Cheng system, this system serves the purpose of the Calabi flow without being fourth order. The benefit of including the term involving $\alpha$ in the first equation is that it allows the flow to converge to cscK metrics that are not K\"ahler-Einstein. See \cite{chen2013pseudo} and references therein for earlier ideas of this type.

Li-Yuan-Zhang \cite{li2020new} proved Shi-type estimates which show that if the flow exists only for a finite time them the Riemann curvature tensor must blow up. They also showed that the flow exists for all time whenever $\Ric(\omega)$ and $\alpha$ are uniformly bounded. In a subsequent paper, Li-Yuan \cite{li2021local} proved a local estimate for the Riemann curvature tensor which is then used to improve the results of \cite{li2020new}. Fei-Guo-Phong \cite{fei2019convergence} studied a generalization of the system where $\alpha$ evolves by $\partial_t \alpha = \kappa \Delta \alpha$, which they refer to as the $\kappa$-LYZ flow. They showed that for $\kappa \neq 1$, estimates for derivatives of all orders follow from uniform estimates on $\omega$ and $\alpha$. Coupled Ricci flows in the real setting have been studied by List \cite{list08} and Guo-Huang-Phong \cite{guo2015pseudo}.

In this paper, we study a system of elliptic equations with a time parameter which are analogous to the $\kappa$-LYZ flow, namely
\begin{align}
    \frac{\omega - \hat \omega}{t} & = -\Ric(\omega) + \lambda \omega + \alpha\label{system1} \\ 
    \frac{\alpha - \hat \alpha}{t} & = \kappa \Delta \alpha\label{system2}
\end{align}
starting from a K\"ahler metric $\hat \omega$ and a closed $(1, 1)$-form $\hat \alpha$. Here $\kappa > 0$ is arbitrary and $\lambda$ is a constant determined by $c_1(X)$ and the cohomology classes of $\hat \omega$, $\hat \alpha$.

This system is related to the parabolic system \eqref{parabolicsystem1}, \eqref{parabolicsystem2} by replacing the time derivatives $\partial_t \omega$, $\partial_t \alpha$ with the secants $(\omega - \hat \omega) / t$, $(\alpha - \hat \alpha) / t$, an idea introduced by La Nave-Tian \cite{nt15} in the context of the K\"ahler-Ricci flow. In their case, a major advantage is that solutions to the continuity equation have an immediate Ricci curvature lower bound. In the work of Zhang-Zhang \cite{zz19}, solutions to the continuity equation are studied on minimal elliptic K\"ahler surfaces and the Ricci curvature lower bound yields a more straightforward proof for the diameter bound. In \cite{zz20}, they study solutions in the case of Fano fibrations. More literature on the continuity equation for K\"ahler metrics can be found in \cite{wondo2023curvature, wondo2023calabi} and for Hermitian metrics in \cite{sw20, shen2024continuity, liang2024continuity}.

One advantage of our system over the $\kappa$-LYZ flow is that all of our results hold with no constraints on $\kappa$. Additionally, the higher order estimates require only a uniform estimate for $\omega$ and not $\alpha$. Due to the elliptic nature of the equations, $C^0$ estimates for the potentials follow easily from the maximum principle. In the case of $c_1(X) < 0$, this is sufficient to prove the required uniform estimate for $\omega$ needed for the higher order estimates. When $c_1(X) = 0$, the potential estimate is only of the form $\| \varphi \|_{C^0} \leq Ct$ which despite not being strong enough to prove the uniform estimate for $\omega$ is nonetheless interesting since the maximum principle for the ordinary complex Monge-Amp\`ere equation gives no information in this case.

Our first result shows that solutions to the system converge smoothly whenever there is a uniform bound on the family of metrics and uniqueness of a cscK metric in the K\"ahler class of the initial metric:

\begin{theorem}[Smooth convergence]\label{convergence}
    Let $\hat \omega$, $\hat \alpha$ be a K\"ahler metric and a closed $(1, 1)$-form satisfying $-c_1(X) + \lambda [\hat \omega] + [\hat \alpha] = 0$ with $\lambda \leq 0$, and assume that $[\hat \omega]$ contains at most one cscK metric. If the system \eqref{system1}, \eqref{system2} satisfies an a priori estimate of the form $C^{-1} \hat \omega \leq \omega(t) \leq C \hat \omega$, then it has a unique solution on $[0, \infty)$ which converges smoothly to the unique cscK metric $\omega_\infty$ in $[\hat \omega]$ and the unique $\omega_\infty$-harmonic $(1, 1)$-form $\alpha_\infty$ in $[\hat \alpha]$.
\end{theorem}

\begin{remark}\label{convergencewhenc1negative}
    If $c_1(X) < 0$ and $\hat \omega \in -r c_1(X)$ for some $r > 0$ or $c_1(X) = 0$ then the hypothesis that $[\hat \omega]$ contains at most one cscK metric is satisfied, since in these cases cscK metrics are K\"ahler-Einstein and are unique in their classes.
\end{remark}

\noindent Combining the above remark with uniform estimates on the family of metrics, we use a special case of the system to show smooth convergence to the unique K\"ahler-Einstein metric in $\frac{1}{\lambda} c_1(X)$, giving an alternative proof of existence:

\begin{theorem}[Smooth convergence to KE metric]\label{KEconvergence}
    Assume $c_1(X) < 0$. Let $\hat \omega$ be a K\"ahler metric satisfying $-c_1(X) + \lambda [\hat \omega] = 0$ with $\lambda < 0$. The equation
    \begin{equation*}
        \frac{\omega - \hat \omega}{t} = -\Ric(\omega) + \lambda \omega
    \end{equation*}
    has a unique solution on $[0, \infty)$ which converges smoothly to the unique K\"ahler-Einstein metric $\omega_\infty$ in $[\hat \omega]$.
\end{theorem}

\noindent In the case of Riemann surfaces of genus at least $2$ which are exactly those with negative first Chern class, we show that solutions to the system starting from more general initial data will smoothly converge to the unique K\"ahler-Einstein metric:

\begin{theorem}[Smooth convergence on Riemann surfaces]\label{riemannsurfacesconvergencefromclass}
    Let $X$ be a Riemann surface with $c_1(X) < 0$, $A \in H^{1,1}(X, \mathbb{R})$ and $\hat \omega$ a K\"ahler metric satisfying \ref{A1}. Then there is an open set of $\hat \alpha \in A$ in the topology of the uniform norm for which the system \eqref{system1}, \eqref{system2} has a unique solution on $[0, \infty)$ which converges smoothly to the unique cscK metric $\omega_\infty$ in $[\hat \omega]$ and the unique $\omega_\infty$-harmonic $(1, 1)$-form $\alpha_\infty$ in $A$.
\end{theorem}

\noindent The technical assumption \ref{A1} is needed for the $C^0$ estimates for $\omega$ and $\alpha$ \eqref{conformalestimates}. This theorem is similar to a result of Fei-Guo-Phong \cite{fei2019convergence}, where the assumption \eqref{A1} is not needed, but $\kappa \neq 1$ is assumed.

This paper is organized as follows. In \cref{section:preliminaries}, we introduce the notation and basic facts that we will use throughout the paper. In \cref{section:thecoupledcontinuityequations}, we introduce the system of elliptic equations, prove its short-time existence, identify its stationary points and prove a scalar curvature lower bound. In \cref{section:reductiontoascalarsystem}, we show that the system reduces to a scalar system involving the complex Monge-Amp\`ere equation, and use this to obtain estimates on potentials. In \cref{section:higherorderestimatesandconvergence}, contingent on a uniform estimate for $\omega$, we prove higher order estimates and smooth convergence to a cscK metric and harmonic $(1, 1)$-form whenever there is a unique cscK metric in the K\"ahler class of the initial metric, proving \cref{convergence}. In \cref{section:kahlereinsteinmetrics}, we establish the required uniform estimate for $\omega$ in a special case by proving a $C^2$ estimate for potentials, which allows us to demonstrate an application of the preceeding convergence theorem and giving a proof of \cref{KEconvergence}. In \cref{section:riemannsurfaces}, we analyze the case of Riemann surfaces and establish the uniform estimate for $\omega$ by estimating conformal factors and demonstrate another application of the convergence theorem, giving a proof of \cref{riemannsurfacesconvergencefromclass}.\\

\noindent\textbf{Acknowledgements:} We would like to thank Teng Fei and Ben Weinkove for some helpful discussions.

\section{Preliminaries}\label{section:preliminaries}

\noindent In this section we establish the notation and basic results that we will use later. Let $X$ be a compact complex manifold of complex dimension $n$.

All tensors considered in this paper are smooth. By $z^1, \dotsc, z^n$ we always mean a local holomorphic coordinate system. Expressions involving indices always refer to such a coordinate system, and repeated indices are always summed over.

Consider a K\"ahler metric $\omega = g_{\bar k j} \, i dz^j \wedge d \bar z^k$. The metric induces an isomorphism $TX \simeq T^* X$, allowing us to raise and lower indices. The volume form of $\omega$ is given locally by
\begin{equation*}
    \frac{\omega^n}{n!} = \det \omega \, i dz^1 \wedge d \bar z^1 \wedge \cdots \wedge i dz^n \wedge d \bar z^n.
\end{equation*}
Since $\omega$ is K\"ahler, its Levi-Civita connection and Chern connection coincide; thus there is a unique torsion-free metric connection $\nabla = \partial + \Gamma$ that is compatible with the complex structure. Its curvature is $R_{\bar k  j}{}^p{}_q = -\partial_{\bar k} \Gamma^p_{j q}$, and this satisfies the Bianchi identity
\begin{equation*}
    R_{\bar k j \bar q p} = R_{\bar k p \bar q j} = R_{\bar q j \bar k p} = R_{\bar q p \bar k j}.
\end{equation*}
Write $\Ric(\omega) = R_{\bar k j} \, i dz^j \wedge d \bar z^k$ where $R_{\bar k j} = R_{\bar k j}{}^p{}_q$. The above identity implies
\begin{equation}\label{equalityofriccis}
    R_{\bar k j} = R^p{}_{p \bar k j}.
\end{equation}
Locally
\begin{equation*}
    \Ric(\omega) = - i \ddb \log \omega^n
\end{equation*}
and consequently $\Ric(\omega)$ is closed. If $\chi$ is another K\"ahler metric then
\begin{equation*}
    \Ric(\omega) - \Ric(\chi) = -i \ddb \log \frac{\omega^n}{\chi^n}
\end{equation*}
so the Ricci form of all K\"ahler metrics lie in the same class $c_1(X)$.

For any closed $(1, 1)$-form $\beta$
\begin{equation}\label{averageoftraceinclass}
    \int_X \tr_\omega \beta \, \frac{\omega^n}{n!} = \int_X [\beta] \wedge \frac{[\omega]^{n - 1}}{(n - 1)!}
\end{equation}
where the notation indicates that this depends only on the cohomology classes $[\omega]$, $[\beta]$.

With respect to $\omega$ the exterior derivative $d = \partial + \bar \partial$ has an adjoint $d^\dagger = \partial^\dagger + \bar \partial^\dagger$. If $K$ is a multi-index, then for $\beta = \beta_{\bar K p} \, dz^p \wedge d \bar z^K$
\begin{equation}\label{adjoint1}
    \partial^\dagger \beta_{\bar K} = -\nabla^p \beta_{\bar K p}\\
\end{equation}
and for $\beta = \beta_{\bar K j p} \, dz^p \wedge dz^j \wedge d \bar z^K$
\begin{equation}\label{adjoint2}
    \partial^\dagger \beta_{\bar K j} = -\nabla^p \beta_{\bar K j p} - g^{r \bar s} \partial^p g_{\bar s j} \beta_{\bar K r p}.
\end{equation}
Define the Laplacian $\Delta = - \frac{1}{2} (d^\dagger d + d d^\dagger)$. It follows that $d \Delta = \Delta d$, and that the Laplacians $-(\partial^\dagger \partial + \partial \partial^\dagger)$ and $-(\bar \partial^\dagger \bar \partial + \bar \partial \bar \partial^\dagger)$ are both equal to $\Delta$. In particular, $\Delta h = \partial^p \partial_p h$ on functions.

A computation shows
\begin{equation*}
    \int_X \nabla_j u^j \, \omega^n = 0
\end{equation*}
for any vector field $u$, and from this it follows that
\begin{equation*}
    \int_X \Delta h \, \omega^n = 0
\end{equation*}
for any function $h$.

If $\lambda > 0$ then the equation $\Delta \beta = \lambda \beta$ implies
\begin{equation*}
    \| d \beta \|^2 + \| d^\dagger \beta \|^2 + \lambda \| \beta \|^2 = 0
\end{equation*}
and hence $\beta = 0$; thus $\Delta$ has no positive eigenvalues.

We will need the following, which we recall from \cite{li2020new}:

\begin{lemma}[Li-Yuan-Zhang \cite{li2020new} Proposition 2.3]\label{ddbcontraction}
    If $\omega$ is a K\"ahler metric and $\beta$ is a closed $(1, 1)$-form then $\Delta \beta = i \ddb \tr_\omega \beta$.
\end{lemma}

\noindent In what follows we will use notation consistent with this section, that is, all unlabeled geometric objects correspond to $\omega$. The value of the constant $C$ changes from line to line, but always remains independent of time and all unknown functions.

\section{The coupled continuity equations}\label{section:thecoupledcontinuityequations}

\noindent Let $\hat \omega$ be a K\"ahler metric, $\hat \alpha$ a closed $(1, 1)$-form, $\lambda \in \mathbb R$, and $\kappa > 0$. We consider the system
\begin{align*}
    \frac{\omega - \hat \omega}{t} & = -\Ric(\omega) + \lambda \omega + \alpha\label{system1}\\
    \frac{\alpha - \hat \alpha}{t} & = \kappa \Delta \alpha.
\end{align*}
which is elliptic for each $t > 0$. Here the unknowns are a K\"ahler metric $\omega(t) = g_{\bar k j}(t) \, i dz^j \wedge d \bar z^k$ and a closed $(1, 1)$-form $\alpha(t) = \alpha_{\bar k j}(t) \, i dz^j \wedge d \bar z^k$. This family of systems is invariant under the scaling
\begin{equation*}
    \omega \mapsto \omega / M, \quad \alpha \mapsto \alpha, \quad t \mapsto t / M, \quad \lambda \mapsto M \lambda, \quad \kappa \mapsto \kappa.
\end{equation*}
Define $\tau = \tr_\omega \alpha$. It will be useful to consider the system
\begin{align}
    \frac{n - \tr_\omega \hat \omega}{t} & = -R + \lambda n + \tau\label{tracesystem1}\\
    \frac{\tau - \tr_\omega \hat \alpha}{t} & = \kappa \Delta \tau\label{tracesystem2}
\end{align}
obtained by taking the trace of \eqref{system1}, \eqref{system2} and applying \cref{ddbcontraction}.

Since $\hat \omega$, $\hat \alpha$, and $\Ric(\omega)$ are closed, solutions to \eqref{system1}, \eqref{system2} satisfy
\begin{align*}
    \frac{d \omega}{t} & = \lambda d \omega + d \alpha\\
    \frac{d \alpha}{t} & = \kappa \Delta d \alpha.
\end{align*}
As $\Delta$ lacks positive eigenvalues, the second equation shows that $d \alpha = 0$. Now the first equation shows that $d \omega = 0$.

Let $\Omega(t)$, $A(t)$ be the cohomology classes of $\omega(t)$, $\alpha(t)$ and write $\hat \Omega = \Omega(0)$, $\hat A = A(0)$. Taking the classes of \eqref{system1}, \eqref{system2} gives
\begin{align}
    \frac{\Omega - \hat \Omega}{t} & = -c_1(X) + \lambda \Omega + A\label{drsystem1}\\
    \frac{A - \hat A}{t} & = 0\label{drsystem2}
\end{align}
since $\Delta \alpha$ is exact by \cref{ddbcontraction}. Thus $A = \hat A$. Solving \eqref{drsystem1} for $(1 - \lambda t) \Omega$ and using $A = \hat A$ results in
\begin{align}\label{combinedcohomologysolution}
    (1 - \lambda t) & (-c_1(X) + \lambda \Omega + A)\\
    & = (1 - \lambda t) (-c_1(X) + A) + \lambda (\hat \Omega - t c_1(X) + t A)\nonumber\\
    & = -c_1(X) + \lambda \hat \Omega + \hat A.\nonumber
\end{align}
We now assume that $-c_1(X) + \lambda \hat \Omega + \hat A = 0$. By \eqref{combinedcohomologysolution} this implies $-c_1(X) + \lambda \Omega + A = 0$. From \eqref{drsystem1} we now see that $\Omega = \hat \Omega$. Thus under this assumption none of the classes vary in time.

Applying \eqref{averageoftraceinclass} to each of $\omega$, $\Ric(\omega)$, $\alpha$ we see that the quantities
\begin{gather*}
    V = \int_X \Omega^n\\
    \underline R = \frac{n}{V} \int_X c_1(X) \wedge \Omega^{n - 1}\\
    \underline \tau = \frac{n}{V} \int_X A \wedge \Omega^{n - 1}
\end{gather*}
depend only on the indicated cohomology classes, where the latter two are the average values of $R$ and $\tau$. By scaling in the way indicated earlier, we may, if convenient, assume that $V = 1$.

By a slight abuse of notation, we will sometimes write $\underline R$ or $\underline \tau$ to indicate the values indicated above even when we are only given the necessary cohomology classes and have not yet chosen representatives. 

The condition $-c_1(X) + \lambda \Omega + A = 0$ allows us to write
\begin{equation}\label{riccipotential1}
    -\Ric(\omega) + \lambda \omega + \alpha = i \ddb F
\end{equation}
for a function $F(t)$; write $\hat F = F(0)$. In particular
\begin{equation}\label{riccipotential2}
    -R + \lambda n + \tau = \Delta F
\end{equation}
and
\begin{equation}\label{averages}
    -\underline R + \lambda n + \underline \tau = 0.
\end{equation}
This shows that
\begin{equation}\label{lambda}
    \lambda = \frac{1}{n} ( \underline R - \underline \tau ) = \frac{1}{V} \int_X (c_1(X) - A) \wedge \Omega^{n - 1}
\end{equation}
We now prove short-time of existence of solutions:
\begin{lemma}[Short-time existence]\label{shorttimeexistence}
    Let $\hat \omega$ be a K\"ahler metric and $\hat \alpha$ a closed $(1, 1)$-form satisfying $-c_1(X) + \lambda [\hat \omega] + [\hat \alpha] = 0$. Assume that $\omega(t)$, $\alpha(t)$ solve \eqref{system1}, \eqref{system2} on $[0, T)$ for some $T \in (0, \infty)$. Then $\omega(t)$, $\alpha(t)$ can be extended to a solution on $[0, T + \varepsilon)$ for some $\varepsilon > 0$.
\end{lemma}

\begin{proof}
    Expanding \eqref{equalityofriccis} shows that the highest order term of $-R_{\bar k j}$ is $\partial^p \partial_p g_{\bar k j}$, and a computation using \eqref{adjoint1}, \eqref{adjoint2} shows that the highest-order term of $\frac{1}{2} \Delta \alpha_{\bar k j}$ is $\partial^p \partial_p \alpha_{\bar k j}$. Thus the map $\mathcal H_+ \times \mathcal H \to \mathcal H_+ \times \mathcal H$ (where $\mathcal H$ is the space of closed $(1, 1)$-forms and $\mathcal H_+$ is the space of K\"ahler metrics) given by
    \begin{equation*}
        \begin{pmatrix} \omega \\ \alpha \end{pmatrix} \mapsto \begin{pmatrix} -\Ric(\omega) + \lambda \omega + \alpha - (\omega - \hat \omega)/t \\ \kappa \Delta \alpha - (\alpha - \hat \alpha)/t \end{pmatrix}
    \end{equation*}
    is strictly elliptic. Elliptic theory thus implies the result.
\end{proof}

\begin{lemma}[Stationary points]\label{stationarypoints}
    A pair $(\omega, \alpha)$ is a stationary point of the system \eqref{system1}, \eqref{system2} (in the sense that it makes the right-hand sides vanish) exactly when $\omega$ is cscK and $\alpha$ is $\omega$-harmonic.
\end{lemma}

\begin{proof}
    Applying \cref{ddbcontraction} to \eqref{riccipotential2} gives
    \begin{equation}\label{riccipotential3}
        -i \ddb R + \Delta \alpha = i \ddb \Delta F.
    \end{equation}
    \begin{enumerate}
        \item Assume $-\Ric(\omega) + \lambda \omega + \alpha = 0$ and $\Delta \alpha = 0$. The former with \eqref{riccipotential1} implies that $F$ is constant, so the latter with \eqref{riccipotential3} implies that $R$ is constant.
        
        \item Assume $\omega$ is cscK and $\alpha$ is $\omega$-harmonic. Then by \eqref{riccipotential3} we have $i \ddb \Delta F = 0$ so $\Delta F$ is constant. Since $\int_X \Delta F \, \omega^n = 0$ actually $\Delta F = 0$ so $F$ is constant. By \eqref{riccipotential1} we have $-\Ric(\omega) + \lambda \omega + \alpha = 0$.\qedhere
    \end{enumerate}
\end{proof}

\noindent We now give a proof of the trace and scalar curvature lower bounds:
\begin{proposition}[Scalar curvature lower bound]\label{traceestimates}
    Let $\hat \omega$, $\hat \alpha$ be a K\"ahler metric and a closed $(1, 1)$-form satisfying $-c_1(X) + \lambda [\hat \omega] + [\hat \alpha] = 0$. Assume that $\omega(t)$, $\alpha(t)$ solve \eqref{system1}, \eqref{system2} on $[0, T)$ for some $T \in (0, \infty]$. If $\hat \alpha \geq 0$ then $\tr_{\omega(t)} \alpha(t) \geq 0$ and
    \begin{equation*}
        R(t) \geq -C \left( 1 + \frac{1}{t} \right).
    \end{equation*}
\end{proposition}
\begin{proof}
    At a minimum point of $\tau$ \eqref{tracesystem2} gives $\tau \geq \tr_\omega \hat \alpha \geq 0$. Now by \eqref{tracesystem1}
    \begin{equation*}
        R = \lambda n + \tau - \frac{n - \tr_\omega \hat \omega}{t} \geq n \left( \lambda - \frac{1}{t} \right).\qedhere
    \end{equation*}
\end{proof}

\section{Reduction to a scalar system}\label{section:reductiontoascalarsystem}

\noindent Since $\Omega = \hat \Omega$, $A = \hat A$ there are functions $\varphi(t)$, $f(t)$ with $\varphi(0) = 0$, $f(0) = 0$ so that
\begin{align}
    \omega & = \hat \omega + i \ddb \varphi\label{potential1}\\
    \alpha & = \hat \alpha + i \ddb f.\label{potential2}
\end{align}
\cref{ddbcontraction} shows that \eqref{system1}, \eqref{system2} become
\begin{align*}
    \frac{i \ddb \varphi}{t} & = i \ddb \left( \log \frac{\omega^n}{\hat \omega^n} + \lambda \varphi + f + \hat F \right)\\
    \frac{i \ddb f}{t} & = \kappa \, i \ddb \tau
\end{align*}
so that (after normalizing by adding time-varying constants to $\varphi$, $f$) we have the scalar system
\begin{align}
    \frac{\varphi}{t} & = \log \frac{(\hat \omega + i \ddb \varphi)^n}{\hat \omega^n} + \lambda \varphi + f + \hat F\label{potentialsystem1}\\
    \frac{f}{t} & = \kappa (\Delta f + \tr_\omega \hat \alpha).\label{potentialsystem2}
\end{align}

\begin{lemma}[A priori estimates for potentials]\label{potentialestimates}
    Let $\hat \omega$, $\hat \alpha$ be a K\"ahler metric and a closed $(1, 1)$-form satisfying $-c_1(X) + \lambda [\hat \omega] + [\hat \alpha] = 0$ with $\lambda \leq 0$. Let $\varphi(t)$, $f(t)$ be functions and define $\omega(t)$, $\alpha(t)$ by \eqref{potential1}, \eqref{potential2}. If $\varphi(t)$, $f(t)$ solve \eqref{potentialsystem1}, \eqref{potentialsystem2} with $\varphi(0) = 0$, $f(0) = 0$ on $[0, T)$ for some $T \in (0, \infty]$ then
    \begin{equation}\label{potentialestimates1}
        \| (1/t - \lambda) \varphi(t) \|_{C^0} \leq \| f(t) \|_{C^0} + \| \hat F \|_{C^0}
    \end{equation}
    and
    \begin{equation}\label{potentialestimates2}
        \| f(t) / t \|_{C^0} \leq \kappa \| { \tr_{\omega(t)} \hat \alpha } \|_{C^0}.
    \end{equation}
\end{lemma}

\begin{proof}
    These follow by applying the maximum principle to \eqref{potentialsystem1}, \eqref{potentialsystem2}. Note that the assumption $\lambda \leq 0$ is used in the first estimate.
\end{proof}

\section{Higher order estimates and convergence}\label{section:higherorderestimatesandconvergence}

\begin{lemma}[Higher order estimates]\label{higherorderestimates}
    Let $\hat \omega$, $\hat \alpha$ be a K\"ahler metric and a closed $(1, 1)$-form satisfying $-c_1(X) + \lambda [\hat \omega] + [\hat \alpha] = 0$ with $\lambda \leq 0$. If the system \eqref{system1}, \eqref{system2} satisfies an a priori estimate of the form
    \begin{equation}\label{uniformequivalence}
        C^{-1} \hat \omega \leq \omega(t) \leq C \hat \omega
    \end{equation}
    then we have the following a priori potential estimates for each $k = 0, 1, \dotsc$
    \begin{equation*}
        \| \varphi(t) \|_{C^k} \leq C_k, \quad  \| f(t) \|_{C^k} \leq C_k.
    \end{equation*}
\end{lemma}

\begin{proof}
    Recall \eqref{potentialsystem1}
    \begin{equation}\label{nonlinearequation}
        \frac{\varphi}{t} = \log \frac{(\hat \omega + i \ddb \varphi)^n}{\hat \omega^n} + \lambda \varphi + f + \hat F.
    \end{equation}
    Applying $\partial / \partial z^\ell$ and restating \eqref{potentialsystem2} results in the linear system
    \begin{gather}
        \frac{\varphi_\ell}{t} = \Delta \varphi_\ell + \lambda \varphi_\ell + f_\ell + g^{j \bar k} \partial_\ell \hat g_{\bar k j} - \hat g^{j \bar k} \partial_\ell \hat g_{\bar k j} + \hat F_\ell\label{linearsystem1}\\
        \frac{f}{t} = \kappa (\Delta f + \tr_\omega \hat \alpha).\label{linearsystem2}
    \end{gather}
    By \eqref{potentialestimates2}, \eqref{uniformequivalence} we obtain a $C^0$ estimate for $f / t$, so we have $C^0$ estimates for the coefficients of \eqref{linearsystem2}. De Giorgi-Nash-Moser theory \cite{gilbarg1977elliptic} gives a $C^\beta$ estimate for $f$ for some $0 < \beta < 1$.

    Now \eqref{potentialestimates1} provides a $C^0$ estimate for $(1/t - \lambda) \varphi$, so we have $C^0$ estimates for the coefficients of \eqref{nonlinearequation}. The regularity result \cite{wang2011remark} gives a $C^{2, \alpha}$ estimate for $\varphi$ for any $0 < \alpha < \beta$.
    
    Therefore we have $C^\alpha$ estimates for the coefficients of \eqref{linearsystem2}, so Schauder estimates give $\| f \|_{C^{2, \alpha}} \leq C$ and hence $\| f_\ell \|_{C^{1, \alpha}} \leq C$. Thus we have $C^\alpha$ estimates for the coefficients of \eqref{linearsystem1}, so Schauder estimates give $\| \varphi_\ell \|_{C^{2, \alpha}} \leq C$ and hence $\| \varphi \|_{C^{3, \alpha}} \leq C$. Continuing in this way results in $C^{k, \alpha}$ estimates of all orders for $\varphi$ and $f$.
\end{proof}

\begin{proof}[Proof of \cref{convergence}]
    By \cref{higherorderestimates} we obtain $C^k$ estimates of all orders on $\varphi$, $f$. These estimates together with \cref{shorttimeexistence} imply, by the method of continuity, that the set of $t$ for which the system is solvable is both open and closed in $[0, \infty)$ and thus that it is the entire interval. After two applications of the Arzela-Ascoli theorem we obtain a sequence $t_j \to \infty$ such that $\varphi(t_j)$, $f(t_j)$ converge smoothly to functions $\varphi_\infty$, $f_\infty$. Let $\omega_\infty = \hat \omega + i \ddb \varphi_\infty$ and $\alpha_\infty = \hat \alpha + i \ddb f_\infty$. Note $\omega_\infty > 0$ since $\omega \geq C^{-1} \hat \omega$; thus $\omega_\infty \in \Omega$ is a K\"ahler metric and $\alpha_\infty \in A$ is a closed $(1, 1)$-form. Since the left-hand sides of \eqref{system1}, \eqref{system2} tend to zero \cref{stationarypoints} implies that $\omega_\infty$ is cscK and $\alpha_\infty$ is $\omega_\infty$-harmonic. By our assumption there is a unique cscK metric in $[\hat \omega]$, and by the Hodge theorem there is a unique $\omega_\infty$-harmonic $(1, 1)$-form in $[\hat \alpha]$. Thus $\omega_\infty$, $\alpha_\infty$ are independent of the choice of subsequence, so in fact $\omega$, $\alpha$ converge smoothly to $\omega_\infty$, $\alpha_\infty$.
\end{proof}

\section{K\"ahler-Einstein metrics}\label{section:kahlereinsteinmetrics}

\noindent In this section we demonstrate how a special case of \cref{convergence} may be used to obtain convergence to a K\"ahler-Einstein metric when one exists. Specifically, we now assume that $-c_1(X) + \lambda \Omega = 0$ and that $\hat \alpha = 0$. Since $\Delta$ lacks positive eigenvalues, \eqref{system2} shows that $\alpha = 0$. Thus the system reduces to a single equation
\begin{equation}\label{continuityequation}
    \frac{\omega - \hat \omega}{t} = -\Ric(\omega) + \lambda \omega
\end{equation}
Similarly, the potential system reduces to
\begin{equation}\label{potentialequation}
    \frac{\varphi}{t} = \log \frac{(\hat \omega + i \ddb \varphi)^n}{\hat \omega^n} + \lambda \varphi + \hat F.
\end{equation}
In order to satisfy the hypothesis of \cref{convergence}, we provide the following standard $C^2$ estimate for the sake of completeness (see e.g. \cite{yau78}). 

\begin{lemma}[$C^2$ estimate]\label{c2estimate}
    Assume $c_1(X) < 0$. Let $\hat \omega$ be a K\"ahler metric satisfying $-c_1(X) + \lambda [\hat \omega] = 0$ with $\lambda < 0$. Let $\varphi(t)$ be a smooth function and define $\omega(t)$ by \eqref{potential1}. If $\varphi(t)$ solves \eqref{potentialequation} with $\varphi(0) = 0$ on $[0, T)$ for some $T \in (0, \infty]$ then
    \begin{equation*}
        C^{-1} \hat \omega \leq \omega(t) \leq C \hat \omega.
    \end{equation*}
\end{lemma}

\begin{proof}
    First observe that since $\hat \alpha = 0$, we may assume that the potential function $f$ is identically zero for all time and so we immediately obtain $C^0$ bounds on $\varphi$ by using \cref{potentialestimates}. This bound is then equivalent to the volume form bound
    \begin{equation*}
        C^{-1} \leq \frac{\omega^n}{\hat \omega^n} \leq C.
    \end{equation*}
    A standard computation using the Aubin-Yau inequality yields
    \begin{equation*}
        \Delta \log \tr_{\hat \omega} \omega \geq -\frac{\hat g^{p\bar q} R_{\bar qp}}{\tr_{\hat \omega} \omega} - C \tr_\omega \hat \omega.
    \end{equation*}
    Using \eqref{continuityequation} gives
    \begin{equation*}
        \Delta \log \tr_{\hat \omega} \omega \geq -C \tr_\omega \hat \omega - C.
    \end{equation*}
    Now $\Delta \varphi = n - \tr_\omega \hat \omega$ so
    \begin{align*}
        \Delta ( \log \tr_{\hat \omega} \omega - A \varphi )
        & \geq (A - C) \tr_\omega \hat \omega - An - C\\
        & \geq \tr_\omega \hat \omega - C.
    \end{align*}
    for $A$ large enough. At a maximum point of $\log \tr_{\hat \omega} \omega - A \varphi$ we thus have $\tr_\omega \hat \omega \leq C$, and by the determinant estimate also $\tr_{\hat \omega} \omega \leq C$. At an arbitrary point this implies
    \begin{equation*}
        \log \tr_{\hat \omega} \omega \leq 2 A \| \varphi \|_{C^0} + C
    \end{equation*}
    so by \cref{potentialestimates} we obtain $\tr_{\hat \omega} \omega \leq C$. Another application of the determinant estimate gives $\tr_\omega \hat \omega \leq C$. These two estimates are equivalent to the desired ones.
\end{proof}

\begin{remark}
    In the case when $c_1(X) = 0$, the maximum principle immediately yields a $C^0$ estimate of the form $\| \varphi \|_{C^0} \leq C t$. This is not strong enough to prove the desired $C^2$ estimate, but it is interesting nonetheless since the maximum principle for the ordinary complex Monge-Amp\`ere equation provides no information in this case.
\end{remark}

\noindent Combining \cref{convergencewhenc1negative} and \cref{c2estimate}, this proves \cref{KEconvergence}.

\section{Riemann surfaces}\label{section:riemannsurfaces}

\noindent Assume $n = 1$. We may write $\omega = e^\phi \hat \omega$ for a function $\phi(t)$ with $\phi(0) = 0$, and we have $\alpha = \tau \omega$. In this setting, the coupled equations \eqref{tracesystem1}, \eqref{tracesystem2} can be written as
\begin{align}
    \frac{1 - e^{-\phi}}{t} & = -R + \lambda + \tau\label{conformalsystem1a}\\
    \frac{\tau - \hat \tau e^{-\phi}}{t} & = \kappa \Delta \tau.\label{conformalsystem2a}
\end{align}
By using $R = \hat R e^{-\phi} - \Delta \phi$ we can alternatively write \eqref{conformalsystem1a} as
\begin{equation}\label{conformalsystem1b}
    \frac{1 - e^{-\phi}}{t} = \Delta \phi - \hat R e^{-\phi} + \lambda + \tau.
\end{equation}
Let us denote $\tau \maxx = \max_X \tau$ and $\tau \minn = \min_X \tau$. In addition, let us use the convention that $h^+ = \max{ \{ h , 0 \} }$ and $h^- = \min{ \{ h , 0 \} }$.

\begin{lemma}[A priori estimates for conformal factors]\label{conformalestimates}
    Let $X$ be a Riemann surface, $\hat \omega$ a K\"ahler metric, $\hat \alpha$ a closed $(1, 1)$-form, $\lambda = \int_X (c_1(X) - A) / \int_X [\hat \omega]$, and assume that $\hat \tau \maxx^+ - \hat \tau \minn^- < -\hat R \maxx$. Let $\phi(t)$, $\tau(t)$ be functions, set $\omega(t) = e^{\phi(t)} \hat \omega$, $\alpha(t) = \tau(t) \omega(t)$, and assume that $\phi(t)$, $\tau(t)$ solve \eqref{conformalsystem1b}, \eqref{conformalsystem2a} with $\phi(0) = 0$ on $[0, T)$ for some $T \in (0, \infty]$. Then $\lambda < 0$ and
    \begin{equation}\label{conformalestimates1}
        \| \phi(t) \|_{C^0} \leq C, \quad \| \tau(t) \|_{C^0} \leq C
    \end{equation}
    so
    \begin{equation}\label{conformalestimates2}
        C^{-1} \hat \omega \leq \omega(t) \leq C \hat \omega, \quad -C \omega(t) \leq \alpha(t) \leq C \omega(t)
    \end{equation}
    and also
    \begin{equation}\label{conformalestimates3}
        \| R(t) \|_{C^0} \leq C \left( 1 + \frac{1}{t} \right).
    \end{equation}
\end{lemma}

\begin{proof}
    First we prove some inequalities on some of the constants that will appear below. The assumption implies $\hat R \maxx - \hat \tau \minn^- < 0$, which is equivalent to
    \begin{gather*}
        \hat R \maxx < 0\\
        \hat R \maxx - \hat \tau \minn < 0.
    \end{gather*}
    By \eqref{lambda}
    \begin{equation*}
        \lambda = \underline R - \underline \tau \leq \hat R \maxx - \hat \tau \minn \leq \hat R \maxx - \hat \tau \minn^- < 0.
    \end{equation*}
    It follows that
    \begin{equation*}
        0 < \frac{\hat R \maxx - \hat \tau \minn^-}{\lambda} \leq 1.
    \end{equation*}
    Equation \eqref{conformalsystem2a} implies the following preliminary estimates for $\tau$
    \begin{align}
        \tau & \geq \hat \tau \minn^- e^{-\phi \minn}\label{tauprelimlowerbound}\\
        \tau & \leq \hat \tau \maxx^+ e^{-\phi \minn}.\label{tauprelimupperbound}
    \end{align}
    By \eqref{conformalsystem1b}, at a minimum point of $\phi$
    \begin{equation*}
        \frac{1 - e^{-\phi}}{t} \geq -\hat R \maxx e^{-\phi} + \lambda + \tau.
    \end{equation*}
    Substituting \eqref{tauprelimlowerbound} gives
    \begin{equation*}
        \frac{1 - e^{-\phi}}{t} \geq -(\hat R \maxx - \hat \tau \minn^-) e^{-\phi} + \lambda.
    \end{equation*}
    Since $\lambda \leq \hat R \maxx - \hat \tau \minn^- < 0$ we obtain the following lower bound for $\phi$
    \begin{equation}\label{philowerbound}
        e^{-\phi} \leq \frac{\lambda}{\hat R \maxx - \hat \tau \minn^-}.
    \end{equation}
    Substituting this into \eqref{tauprelimlowerbound}, \eqref{tauprelimupperbound} establishes
    \begin{equation}\label{taubounds}
        \frac{\lambda \hat \tau \minn^-}{\hat R \maxx - \hat \tau \minn^-} \leq \tau \leq \frac{\lambda \hat \tau \maxx^+}{\hat R \maxx - \hat \tau \minn^-}.
    \end{equation}
    By \eqref{conformalsystem1b}, at a maximum point of $\phi$
    \begin{equation*}
        \frac{1 - e^{-\phi}}{t} \leq -\hat R \minn e^{-\phi} + \lambda + \tau.
    \end{equation*}
    Substituting \eqref{taubounds} gives
    \begin{equation*}
        \frac{1 - e^{-\phi}}{t} \leq -\hat R \minn e^{-\phi} + \lambda + \frac{\lambda \hat \tau \maxx^+}{\hat R \maxx - \hat \tau \minn^-}.
    \end{equation*}
    Since $\hat R \minn < 0$ and (by our assumption) $1 + \hat \tau \maxx^+ / (\hat R \maxx - \hat \tau \minn^-) > 0$ we obtain the following upper bound for $\phi$
    \begin{align*}
        e^{-\phi}
        & \geq \frac{\lambda + \lambda \hat \tau \maxx^+ / (\hat R \maxx - \hat \tau \minn^-) - 1/t}{\hat R \minn - 1/t}\\
        & \geq \min{ \left\{ \frac{\lambda + \lambda \hat \tau \maxx^+ / (\hat R \maxx - \hat \tau \minn^-)}{\hat R \minn} , 1 \right\} }.
    \end{align*}
    The scalar curvature estimates follow from \eqref{conformalsystem1a}.
\end{proof}

\begin{remark}
    Only the estimates $\phi \leq C$, $\omega \leq C \hat \omega$, $R \geq -C ( 1 + 1/t )$ need the full strength of this assumption; the others need only $\hat R \maxx - \hat \tau \minn^- < 0$.
\end{remark}

\begin{theorem}[Smooth convergence]\label{riemannsurfacesconvergence}
    Let $X$ be a Riemann surface with $c_1(X) < 0$, $\hat \omega$ a K\"ahler metric, $\hat \alpha$ a closed $(1, 1)$-form, $\lambda = \int_X (c_1(X) - [\hat \alpha]) / \int_X [\hat \omega]$, and assume that $\hat \tau \maxx^+ - \hat \tau \minn^- < -\hat R \maxx$. The system \eqref{system1}, \eqref{system2} has a unique solution on $[0, \infty)$ which converges smoothly to the unique cscK metric $\omega_\infty$ in $[\hat \omega]$ and the unique $\omega_\infty$-harmonic $(1, 1)$-form $\alpha_\infty$ in $[\hat \alpha]$.
\end{theorem}

\begin{proof}
    Since $c_1(X) < 0$ and $n = 1$, it follows that $\hat \omega \in -r c_1(X)$ where $r = -\int_X [\hat \omega] / \int_X c_1(X)$. \cref{conformalestimates} gives $C^{-1} \hat \omega \leq \omega \leq C \hat \omega$ as well as $\lambda < 0$, so the result follows from \cref{convergencewhenc1negative}.
\end{proof}

\noindent In the final theorem, we show that the assumption $\hat \tau \maxx^+ - \hat \tau \minn^- < - \hat R \maxx$ required for \cref{riemannsurfacesconvergence} is not restrictive. We will need the techinical assumption that 
\begin{align}
    \hat R \maxx < \underline \tau < -\hat R \maxx ,\tag{A1}\label{A1}
\end{align}
where $\underline \tau = \frac{\int_X \alpha}{\int_X \omega}$ and $\hat R \maxx$ is the maximum scalar curvature of the initial metric $\hat\omega.$
This assumption is needed for technical reasons in the proof of estimates for the $\omega$ and $\alpha$ in \ref{conformalestimates2}. 

\begin{proof}[Proof of \cref{riemannsurfacesconvergencefromclass}]
    Pick any $\hat \alpha \in A$. Because $\hat R \maxx < \underline \tau < -\hat R \maxx$, we can compress $\hat \tau$ towards its average value until we obtain $\frac{1}{2} \hat R \maxx < \hat \tau < -\frac{1}{2} \hat R \maxx$ everywhere, and doing this preserves the condition $\hat \alpha \in A$. These inequalities imply $\hat \tau \maxx^+ - \hat \tau \minn^- < -\hat R \maxx$, so the result follows from \cref{riemannsurfacesconvergence}.
\end{proof}

\bibliography{references}

\end{document}